\documentclass[12pt]{amsart}
\usepackage{amssymb}
\usepackage{amsmath}
\makeatletter
\renewcommand*\env@matrix[1][*\c@MaxMatrixCols c]{%
  \hskip -\arraycolsep
  \let\@ifnextchar\new@ifnextchar
  \array{#1}}
\makeatother
\usepackage{amsthm}
\newtheorem{innercustomthm}{Theorem}
\newenvironment{customthm}[1]
  {\renewcommand\theinnercustomthm{#1}\innercustomthm}
  {\endinnercustomthm}
\usepackage{xcolor}

\newtheorem{dfn}{Definition}[section]
\newtheorem{thrm}[dfn]{Theorem}

\newtheorem{lm}[dfn]{Lemma}

\newtheorem{cor}[dfn]{Corollary}

\newcommand{\rem}{\noindent{\bf Remark.  }}

\numberwithin{equation}{section}

\title
[An Ambarzumian type theorem on graphs]{Spectral determinants and an Ambarzumian type theorem on graphs}

\author{M\'arton Kiss}
\thanks{2010. Math. Subject Classification: Primary 34A55, 34B20, 34B24, 34B45;
Secondary 34L40, 47A75\newline Key words and phrases: Ambarzumian, inverse problems, inverse eigenvalue problem, differential equations on graphs, quantum graphs, spectral determinant, Matrix Tree Theorem}
\address
{Department of Differential Equations
\newline
\indent Institute of Mathematics
\newline
\indent Budapest University of Technology and Economics
\newline
\indent H 1111 Budapest, M\H{u}egyetem rkp. 3-9.}
\email{mkiss@math.bme.hu}
\begin{document}

\begin{abstract}
We consider an inverse problem for Schr\"odinger operators on connected equilateral graphs with standard matching conditions. We calculate the spectral determinant and prove that the asymptotic distribution of a subset of its zeros can be described by the roots of a polynomial. We verify that one of the roots is equal to the mean value of the potential and apply it to prove an Ambarzumian type result, i.e., if a specific part of the spectrum is the same as in the case of zero potential, then the potential has to be zero.
\end{abstract}
\maketitle
\begin{section}{Introduction}

Quantum graphs arise naturally as simplified models in mathematics, physics, chemistry, and engineering \cite{BerkolaikoKuchment2013book}. Ambarzumian's theorem in inverse spectral theory refers to a setting when a differential operator can be reconstructed from at most one spectrum due to the presence of a constant eigenfunction. The original theorem from 1929 states for $q\in C[0,\pi]$ that if the eigenvalues of
\begin{align}%
\left.
\begin{array}{cc}
-y''+q(x)y=\lambda y\\
y'(0)=y'(1)=0       
\end{array}
\right\}
\end{align}
are $\lambda_n=n^2\pi^2$ ($n\ge 0$), then $q=0$ \cite{A}. Eigenvalues other than zero are used only through eigenvalue asymptotics to get $\int_0^{1}q=0$; hence a subsequence $\lambda_r=r^2\pi^2+o(1)$ of them is sufficient to reach the same conclusion even if $q\in L^1(0,\pi)$. On finite intervals inverse eigenvalue problems have a vast literature. In general we mention \cite{Horvath2005Annals} and the basic paper \cite{B1} as well as the works referenced by and referencing these. For Ambarzumian's theorem a recent stability result is found in \cite{Horvath2015}. Let us turn to the list of extensions to graphs. On a tree with edges of equal length knowing the smallest eigenvalue $0$ exactly and a specific part of the spectrum approximately is enough for recover to the zero potential \cite{CarlsonPivovarchik2007ambarzumian}. On a tree with different edge lengths this is still true (see Lemma 4.4 of \cite{LawYanagida2012}), however, the required set of eigenvalues is given by an existence proof.
On the other hand, having information about the entire spectrum allows applying a trace formula, even if a detailed description of the structure of the spectrum is not available. The paper \cite{Davies2013} uses heat extension in an abstract framework; this allows arbitrary graphs with arbitrary edge lengths at the expense of requiring information from the entire spectrum. %

For a summary on differential operators on graphs, see \cite{PokornyiBorovskikh2004,Kuchment2008}. %

In this paper we consider a connected graph $G(V,E)$ with edges of equal length. The graph can contain loops and multiple edges. We parametrize each edge with $x\in(0,1)$. This gives an orientation on $G$. We consider a Schr\"odinger operator with potential $q_j(x)\in L^1(0,1)$ on the edge $e_j$ and with Neumann (or Kirchhoff) boundary conditions (sometimes called standard matching conditions), i.e., solutions are required to be continuous at the vertices and, in the local coordinate pointing outward, the sum of derivatives is zero. 
More formally, consider the eigenvalue problem
\begin{align}\label{sch}
-y''+q_j(x)y=\lambda y
\end{align}
on $e_j$ for all $j$ with the conditions
\begin{align}\label{continuity}
y_j(\kappa_j)=y_k(\kappa_k)
\end{align}
if $e_j$ and $e_k$ are incident edges attached to a vertex $v$ where $\kappa=0$ for outgoing edges, $\kappa=1$ for incoming edges (and can be both $0$ or $1$ for loops); and in every vertex $v$
\begin{align}\label{Kirchhoff}
\sum_{e_j\textrm{ leaves }v} y_j'(0)=\sum_{e_j\textrm{ enters }v} y_j'(1)
\end{align}
(loops are counted on both sides).

The \emph{spectral determinant} or alternatively \emph{functional determinant} or \emph{characteristic function} of the problem (\ref{sch})-(\ref{Kirchhoff}) is a meromorphic function whose zeros coincide with its spectrum. Spectral determinants have been subject to continuous attention in the theoretical physics literature for the last twenty years \cite{CurrieWatson2005, KacPivovarchik2011, CarlsonPivovarchik2008SpectralAsymptotics, Pankrashkin2006, AkkermansComtetDesboisMontambauxTexier2000, Desbois2000, Friedlander2006, HarrisonKirstenTexier2012, Texier2010}. As a tool for proving our Ambarzumian type results, we give a formula for the spectral determinant of Schr\"odinger operators (see (\ref{detM})) and for the asymptotic distribution of some of its zeros. It is already known that for finite connected graphs, the main term of the eigenvalue asymptotics can be obtained from Weyl's law and the next terms depend on complicated combinations of $\int_0^1q_j$, $j=1,\ldots,|E|$ (\cite{MollerPivovarchik2015book}, p. 213). We express some of these combinations as the roots of a polynomial (see (\ref{eigpol})) and prove that for \emph{any} connected, equilateral graph there is a root equal to the mean value of the potential. For the convenience of the reader, we emphasize this result in the context of spectral determinants:

\begin{customthm}{1.1'}\label{spectraldet}
Consider a connected graph $G(V,E)$ with edges of equal length. The spectral determinant of Schr\"odinger operators on $G$ with standard matching conditions has a sequence of roots which asymptotically differ by the mean value of the potential from the corresponding sequence of roots of the spectral determinant of the free Schr\"odinger operator. Precisely, the problem (\ref{sch})-(\ref{Kirchhoff}) has a sequence of eigenvalues $\lambda_k=(2k\pi)^2+\frac{1}{|E|}\sum_{j}\int_0^{1}q_j+o(1)$ ($k\in\mathbb{Z}^+$). Moreover, if $G$ is a bipartite graph, the problem (\ref{sch})-(\ref{Kirchhoff}) has a sequence of eigenvalues $\lambda_k=(k\pi)^2+\frac{1}{|E|}\sum_{j}\int_0^{1}q_j+o(1)$.
\end{customthm}

\rem
It would be interesting to prove Theorem \ref{spectraldet} in the case of possibly different edge lengths and to describe the asymptotic distribution of the eigenvalues in question.

For our Ambarzumian type result, we need an improved version of this theorem:

\begin{thrm}\label{equallengthdirect}
Consider the eigenvalue problem (\ref{sch})-(\ref{Kirchhoff}). There are exactly $|E|-|V|+2$ eigenvalues (counting multiplicities) such that $\lambda=(2k\pi)^2+O(1)$ as $k\to\infty$ ($k\in\mathbb{Z}^+$). Among these eigenvalues at least one has the asymptotics $\lambda=(2k\pi)^2+\frac{1}{|E|}\sum_{j}\int_0^{1}q_j+o(1)$. Moreover, if $G$ is a bipartite graph, the same is true for $k$ instead of $2k$, i.e., there are exactly $|E|-|V|+2$ eigenvalues (counting multiplicities) such that $\lambda=(k\pi)^2+O(1)$ and at least one has the asymptotics $\lambda=(k\pi)^2+\frac{1}{|E|}\sum_{j}\int_0^{1}q_j+o(1)$.
\end{thrm}

\rem
Although $G$ is a directed graph, if we reverse an edge $e_j$ and change the potential to $q_j(1-x)$ on it, we get an eigenvalue problem with the same eigenvalues and eigenfunctions (which are reversed with respect to the new direction on $e_j$).

\rem
For $q=0$ and $\lambda=2k^2\pi^2$ ($k\in\mathbb{Z}^+$)
one eigenfunction is $\cos{2k\pi x}$ and in each circle there is a Dirichlet eigenfunction $\pm\sin{2k\pi x}$ on its edges (depending on the direction) and $0$ everywhere else. In the bipartite case we can assume that $V$ is a disjoint union of $V_1$ and $V_2$, and that every edge points from $V_1$ to $V_2$. %
Then for $k$ odd, we can take $\cos{k\pi x}$ on all edges; besides, there are Dirichlet eigenfunctions, $\pm\sin{k\pi x}$ alternately on the edges of a circle and $0$ elsewhere. Hence in both cases there are $|E|-|V|+2$ independent eigenfunctions. We shall prove that the multiplicity is not greater.

\begin{thrm}\label{equallengthinverse}
Consider the eigenvalue problem (\ref{sch})-(\ref{Kirchhoff}). If $\lambda=0$ is the smallest eigenvalue and for infinitely many $k\in\mathbb{Z}^+$ there are $|E|-|V|+2$ eigenvalues (counting multiplicities) such that $\lambda=(2k\pi)^2+o(1)$, then $q=0$ a.e. on $G$. Moreover, if $G$ is a bipartite graph, the same is true for $k$ instead of $2k$, i.e., if $\lambda=0$ is the smallest eigenvalue and for infinitely many $k\in\mathbb{Z}^+$ there are $|E|-|V|+2$ eigenvalues (counting multiplicities) such that $\lambda=(k\pi)^2+o(1)$, then $q=0$ a.e. on $G$.
\end{thrm}

\rem
For a tree $|E|-|V|+2=1$, and a tree is bipartite, hence Theorem \ref{equallengthinverse} is a generalization of Theorem 1.2 in \cite{CarlsonPivovarchik2007ambarzumian}.

If the graph represents an electrical circuit in which each edge has a unit resistance, the \emph{effective resistance} of an edge can be computed (or in mathematics, defined) by terms of the graph Laplacian. A result of Kirchhoff \cite{Kirchhoff1847} is that the effective resistance of an edge $e$ can be expressed as the number of spanning trees containing $e$ divided by the number of all spanning trees.
\begin{thrm}\label{equallengthequalresistances}
Consider the eigenvalue problem (\ref{sch})-(\ref{Kirchhoff}). If every non-loop edge of $G$ has the same effective resistance $r<1$, the multiplicities required by Theorem \ref{equallengthinverse} can be weakened to $|E|-|V|+1$.
\end{thrm}

Interesting examples are the complete graph (if $|V|>2$) or a graph with one point and a loop, which corresponds to the case of periodic boundary conditions treated in \cite{ChengWangWu2010}.

\end{section}

\begin{section}{The proof}

Denote by $c_j(x,\lambda)$ the solution of (\ref{sch}) which satisfies the conditions
$c_j(0,\lambda)-1 =c_j'(0,\lambda) = 0$ and by $s_j(x,\lambda)$ the solution of (\ref{sch}) which satisfies the conditions $s_j(0,\lambda) =s_j'(0,\lambda)-1 = 0$. %
Each $y_j(x,\lambda)$ may be written as a linear combination
\begin{align}
y_j(x,\lambda)=A_j(\lambda)c_j(x,\lambda)+\tilde B_j(\lambda)s_j(x,\lambda).
\end{align}
Then $y_j(0,\lambda)=A_j(\lambda)$ is the same on each outgoing edge; hence we choose to index the functions $A(\lambda)$ by vertices, and then
\begin{align}\label{eigfun}
y_j(x,\lambda)=A_v(\lambda)c_j(x,\lambda)+\tilde B_j(\lambda)s_j(x,\lambda),
\end{align}
if $e_j$ starts from $v$. If the eigefunctions are normalized, i.e., $\sum_j\|y_j(x,\lambda)\|_2^2=1$, then $A_v(\lambda)=O(1)$, $\tilde B_j(\lambda)=O(\sqrt{\lambda})$ (\cite{CarlsonPivovarchik2007ambarzumian}). For later calculations it is more convenient to work with the $O(1)$-variables $A_v(\lambda)$ and $B_j(\lambda)=\frac{\tilde B_j(\lambda)}{\sqrt{\lambda}}$.

The coefficients $A_v$ and $B_j$ form a $(|V|+|E|)$-vector, which satisfies $|V|$ Kirchhoff conditions at the vertices and $|E|$ continuity conditions at the incoming ends of edges, namely, for all $v\in V(G)$,
\begin{align}%
\sum_{e_j:\ldots\to v} \frac{1}{\sqrt{\lambda}}A_{v_j}(\lambda)c'_j(1,\lambda)+B_j(\lambda)s_j'(1,\lambda)-\sum_{e_j:v\to\ldots}B_j(\lambda)=0,
\end{align}
where in the first sum $v_j$ denotes the starting point of $e_j$; and for all $e_j\in E(G)$,
\begin{align}%
A_u(\lambda)c_j(1,\lambda)+\sqrt{\lambda}B_j(\lambda)s_j(1,\lambda)-A_v(\lambda)&=0,
\end{align}
if $e_j$ points from $u$ to $v$.

The matrix of this homogeneous linear system of equations has the form $M=\left[\begin{array}{cc}A&B\\C&D\end{array}\right]$, where
\begin{itemize}
\item $A$ is like an adjacency matrix; $a_{vu}=\frac{1}{\sqrt{\lambda}}\sum c_j'(1,\lambda)$, the sum is taken on edges pointing from $u$ to $v$;
\item $B$ and $C$ are like incidence matrices;\\
$b_{vj}=\left\{\begin{array}{ccc}s_j'(1,\lambda)&\textrm{ if $e_j$ ends in $v$}\\-1&\textrm{ if $e_j$ starts from $v$}\\s_j'(1,\lambda)-1&\textrm{ if $e_j$ is a loop in $v$}\\0&\textrm{otherwise}\end{array}\right.$\\
$c_{jv}=\left\{\begin{array}{ccc}-1&\textrm{ if $e_j$ ends in $v$}\\c_j(1,\lambda)&\textrm{ if $e_j$ starts from $v$}\\-1+c_j(1,\lambda)&\textrm{ if $e_j$ is a loop in $v$}\\0&\textrm{otherwise}\end{array}\right.$
\item $D$ is a diagonal matrix, $d_{jj}=\sqrt{\lambda}s_j(1,\lambda)$.
\end{itemize}

The determinant of the matrix $M$ is the so-called spectral determinant of the problem (\ref{sch})-(\ref{Kirchhoff}).

\noindent{\bf Example 1. }
Consider a single vertex with a loop. Then
\begin{equation*}
M=M_1=\left[\begin{array}{c|c}
  \frac{1}{\sqrt{\lambda}}c'(1,\lambda)&s'(1,\lambda)-1\\\hline
  -1+c(1,\lambda)&\sqrt{\lambda}s(1,\lambda)
\end{array}\right].
\end{equation*}

\noindent{\bf Example 2. }
Consider a star graph with root $r$ and vertices $u$, $v$ and $w$. Let $e_1$, $e_2$ and $e_3$ point from $u$, $v$ and $w$ to $r$, respectively. We choose to index rows and columns by $r$, $u$, $v$, $w$, $e_1$, $e_2$, $e_3$, in that order. Then the matrix $M=M_2$ is:
\begin{equation*}
\left[\begin{array}{cccc|ccc}
  0&\mkern-12mu\frac{1}{\sqrt{\lambda}}c_1'(1,\lambda)&\mkern-12mu\frac{1}{\sqrt{\lambda}}c_2'(1,\lambda)&\mkern-12mu\frac{1}{\sqrt{\lambda}}c_3'(1,\lambda)&s_1'(1,\lambda)&s_2'(1,\lambda)&s_3'(1,\lambda)\\
  0&0&0&0&-1&0&0\\
  0&0&0&0&0&-1&0\\
  0&0&0&0&0&0&-1\\\hline
  \mkern-6mu-1&c_1(1,\lambda)&0&0&\mkern-6mu\sqrt{\lambda}s_1(1,\lambda)&0&0\\
  \mkern-6mu-1&0&c_2(1,\lambda)&0&0&\mkern-18mu\sqrt{\lambda}s_2(1,\lambda)&0\\
  \mkern-6mu-1&0&0&c_3(1,\lambda)&0&0&\mkern-18mu\sqrt{\lambda}s_3(1,\lambda)\mkern-6mu
\end{array}\right]\mkern-6mu,
\end{equation*}
with determinant $\frac{1}{\sqrt{\lambda}}\sum_{j=1}^{3}c_j'(1,\lambda)\prod_{p\ne j}c_p(1,\lambda)$ (corresponding to formula (5) of \cite{Pivovarchik2005ambarzumian}).

The elements of $M$ have the following asymptotics for $\lambda=k^2\pi^2+d+o(1)$ (see \cite{CarlsonPivovarchik2007ambarzumian} eq. (2.3) or \cite{LawYanagida2012} Lemma 3.1):
\begin{align}
\frac{1}{\sqrt{\lambda}}c_j'(1,\lambda)&=(-1)^k\frac{1}{2\sqrt{\lambda}}(\int_0^{1}q_j-d)+o(\frac{1}{\sqrt{\lambda}}),\label{asym1}\\
s_j'(1,\lambda)&=(-1)^k+o(\frac{1}{\sqrt{\lambda}}),\label{asym2}\\
c_j(1,\lambda)&=(-1)^k+o(\frac{1}{\sqrt{\lambda}}),\label{asym3}\\
\sqrt{\lambda}s_j(1,\lambda)&=(-1)^k\frac{1}{2\sqrt{\lambda}}(d-\int_0^{1}q_j)+o(\frac{1}{\sqrt{\lambda}}).\label{asym4}
\end{align}

\rem
Using these asymptotics, we get $\det M_1=c'(1,\lambda)s(1,\lambda)+o(\frac{1}{\lambda})$ for $k$ even (using the Wronskyan would yield only $o(\frac{1}{\sqrt{\lambda}})$), and 
$\det M_2=\frac{1}{\sqrt{\lambda}}\sum_{j=1}^{3}c_j'(1,\lambda)+o(\frac{1}{\sqrt{\lambda}})$, for all $k$. These are special cases of (\ref{detM}) below.

\begin{lm}\label{incidence}
If $\lambda=(2k)^2\pi^2+O(1)$, or $\lambda=k^2\pi^2+O(1)$ and $G$ is bipartite, then every $|V|\times |V|$ submatrix of $C$ (and of $B$) has determinant at most $o(\frac{1}{\sqrt{\lambda}})$.
\end{lm}
\begin{proof}
Leaving out the $o(\frac{1}{\sqrt{\lambda}})$ terms from the submatrix we make only $o(\frac{1}{\sqrt{\lambda}})$ error in its determinant. What we get is an incidence matrix of a graph with $|V|$ vertices and $|V|$ edges. This must contain a circle, hence the corresponding rows are dependent. The proof for $B$ is similar.
\end{proof}

\begin{lm}
The determinant of $M$ is $O(\lambda^{-\frac{1}{2}(|E|-|V|+2)})$.
\end{lm}
\begin{proof}
Look at the terms in the Laplace expansion.
Taking $(|V|-1)$ factors from $B$ (and consequently from $C$) we have to take $(|E|-|V|+2)$ factors of magnitude $O(\frac{1}{\sqrt{\lambda}})$ from $A$ and $D$; otherwise, we get smaller terms, using the previous lemma.
\end{proof}
The next statement is a variant of the Matrix Tree Theorem (\cite{Kirchhoff1847}; see also \cite{Lovasz1993Combinatorial}, p. 252, \cite{Bollobas1998Modern}, Theorem II.12, \cite{Tutte2001book}, Theorem VI.29).

\begin{thrm}\label{direct}
If $\lambda=k^2\pi^2+O(1)$ and $k$ is even or $G$ is bipartite, then
\begin{align}\label{detM}
\det M=(-1)^{k|V|}\sum_{\tau}\!\left(\!\frac{1}{\sqrt{\lambda}}\sum_{e_j\in G} c_j'(1,\lambda)\!\right)\!\!\prod_{e_j\notin\tau}\sqrt{\lambda}s_j(1,\lambda)+o(\lambda^{-\frac{|E|-|V|+2}{2}}).
\end{align}
where the sum is taken for all spanning trees $\tau$ of $G$.
\end{thrm}
\begin{proof}
The main terms in the Laplace expansion are those which contain exactly $(|E|-|V|+1)$ elements from $D$. The product of a fixed set of $(|E|-|V|+1)$ elements in $D$ is weighted by the determinant of the respective minor, with all other elements of $D$ substituted by zero. The remaining rows in $C$ and columns in $B$ look like an ordered (or unordered) incidence matrix of the graph $\tau$ spanned by the remaining $(|V|-1)$ edges for $k$ even (or odd, respectively). If $\tau$ contains a circle, then the determinant of the minor is $o(\frac{1}{\sqrt{\lambda}})$. Otherwise $\tau$ is a spanning tree of $G$; then the determinant is the sum of the elements in $A$ (plus $o(\frac{1}{\sqrt{\lambda}})$), as it follows from the next two lemmas.
\end{proof}
\begin{lm}
Let $\tau$ be a spanning tree of $G$ and $R$ the ordered vertex-edge incidence matrix for $\tau$. Consider the matrix $M_1=\left[\begin{array}{cc}X&R\\-R^T&Y\end{array}\right]$, where $Y$ is a $(|V|-1)\times(|V|-1)$ zero matrix and $X$ has only one nonzero element $s$. Then the determinant of $M_1$ is $s$, independently of the position of the nonzero element in $X$.
\end{lm}
\begin{proof}
In the Laplace expansion every nonzero term (in fact there is only one) contains an element from $X$, hence it is enough to prove this for $s=1$. Let the indices of the nonzero element in $X$ be $uv$. First we prove that the determinant of $M_1$ is independent of $v$. Indeed, there is a path in the tree between two arbitrary vertices, hence we can add to (or subtract from) row $u$ the rows corresponding to the vertices of that path. Similarly, the determinant does not depend on $u$. Reversing an edge in $G$ does not change the determinant; using this and by adding rows (and corresponding columns) we can assume that $\tau$ is a path. Then taking $u=|V|$, $v=1$ and expanding the determinant from left to the right we get $((-1)^{|V|)})^{|V|-1}=1$.
\end{proof}
\begin{lm}
Let $G$ be a bipartite graph, $\tau$ a spanning tree of $G$ and $R$ the unordered incidence matrix for $\tau$. Consider the matrix $M_1=\left[\begin{array}{cc}X&-R\\-R^T&Y\end{array}\right]$, where $Y$ is a $(|V|-1)\times(|V|-1)$ zero matrix and $X$ has only one nonzero element, $x_{uv}=s$, such that $(uv)\in E(G)$. Then the determinant of $M_1$ is $(-1)^{|V|}s$, independently of $u$ and $v$.
\end{lm}
\begin{proof}
There is only one nonzero term in the Laplace expansion of the determinant, which contains only $\pm1$'s besides $s$, hence $\det M_1=\pm s$. If $V$ is a disjoint union of $V_1$ and $V_2$ such that all edges connect $V_1$ to $V_2$, then let us multiply by $(-1)$ the rows in $R$ corresponding to $V_1$ and the columns in $R^T$ corresponding to $V_2$, respectively. This multiplies the determinant by $(-1)^{|V|}$ leaving the nonzero element of $X$ unchanged for $(uv)\in E(G)$. Hence the statement follows from the previous lemma.
\end{proof}
Substituting the asymptotics (\ref{asym1})-(\ref{asym4}) we get
\begin{cor}
If $\lambda=k^2\pi^2+d+o(1)$ and $k$ is even or $G$ is bipartite, then
\begin{align}
\det M=(-1)^{k|E|}\left(\frac{1}{2\sqrt{\lambda}}\right)^{|E|-|V|+2}p(d)+o(\lambda^{-\frac{1}{2}(|E|-|V|+2)}),
\end{align}
where
\begin{align}\label{eigpol}
p(d)=\sum_{\tau}\left(\sum_{e_j\in G}\int_0^{1}q_j-|E|d\right)\prod_{e_j\notin\tau}(d-\int_0^{1}q_j),
\end{align}
the outer sum is taken for all spanning trees $\tau$ of $G$.
\end{cor}

\noindent{\bf Proof of Theorem \ref{equallengthdirect}.}

$\lambda$ is an eigenvalue of the eigenvalue problem (\ref{sch})-(\ref{Kirchhoff}) if and only if $\det M(\lambda)=0$. Let the distinct roots of $p(d)$ be $d_1,\ldots,d_l$. By the previous corollary for $\lambda=k^2\pi^2+O(1)$ (if $k$ is even or $G$ is bipartite) the distinct roots of $\det M(\lambda)$ are exactly of the form $\lambda=k^2\pi^2+d_j+o(1)$ $(1\le j\le l)$. As it is seen from (\ref{eigpol}), one of them is $\lambda=k^2\pi^2+\frac{1}{|E|}\sum_{e_j\in G}\int_0^{1}q_j+o(1)$.

It remains to prove that the total multiplicities of the eigenvalues $\lambda=k^2\pi^2+O(1)$ are exactly $|E|-|V|+2$. A direct calculation shows that this is true for $q=0$. Indeed, then $\det M$ is a polynomial of $\cos{\sqrt{\lambda}}$ and $\sin{\sqrt{\lambda}}$, hence its zeros are $2\pi$-periodic (in the bipartite case $\pi$-periodic) in $\sqrt{\lambda}$.
Hence $\lambda=k^2\pi^2+O(1)\implies\sqrt{\lambda}=k\pi+o(1)$ implies $\sqrt{\lambda}=k\pi$ with finitely many exceptions. For $\lambda=k^2\pi^2$ $A$ and $D$ are zero matrices, thus the rank of $M$ is $2(|V|-1)$, and its nullspace is exactly $(|E|-|V|+2)$-dimensional. Then consider the eigenvalue problem $-y''+tq_j(x)y=\lambda y$ with the boundary conditions (\ref{continuity})-(\ref{Kirchhoff}) for $t\in[0,1]$. The corresponding operator $T(t)$ forms a self-adjoint holomorphic family and its eigenvalues $\lambda_n(t)$ and normalized eigenfunctions $g_n(t)$ can be represented by holomorphic functions of $t$ (see Example VII-3.5 and Theorem VII-3.9 in \cite{K}).
\begin{align}
\lambda_n'(t)=\langle g_n(t),T'(t)g_n(t)\rangle=\int_Gg_n^2(t)q 
\end{align}
is bounded by $\|g_n\|_{\infty}^2\|q\|_{1}=O({\|q\|_{1}})$ (see (\ref{eigfun}) and the paragraph below it). Hence $|\lambda_n(1)-\lambda_n(0)|=O(1)$ and the total multiplicity of eigenvalues $\lambda=k^2\pi^2+O(1)$ is the same for all $q\in L^1$.
\qed

\noindent{\bf Proof of Theorem \ref{equallengthinverse}.}

According to Theorem \ref{equallengthdirect},
\begin{align}
\int_Gq=\sum_{e_j\in G}\int_0^{1}q_j=0. 
\end{align}
Let us denote the operator of the eigenvalue problem (\ref{sch})-(\ref{Kirchhoff}) by $L$. $\langle \varphi,L\varphi\rangle\ge\lambda_0=0$ and equality holds if and only if $\varphi$ is an eigenfunction of $L$. It follows that the constant $1$ must be an eigenfunction corresponding to the eigenvalue $0$. Substituting this to (\ref{sch}) gives $q(x)=0$.
\qed

\noindent{\bf Proof of Theorem \ref{equallengthequalresistances}.}

The equality of the effective resistances $r<1$ implies that the number of spanning trees not containing a fixed edge is the same nonzero integer for every non-loop edge.
Suppose that $\sum_{e_j\in G}\int_0^{1}q_j\ne0$. Then by (\ref{eigpol})
\begin{align}
\prod_{e_j\textrm{is a loop}}(d-\int_0^{1}q_j)\sum_{\tau}\prod_{\substack{e_j\textrm{is not a loop}\\e_j\notin\tau}}(d-\int_0^{1}q_j)=\sum_{\tau}d^{|E|-|V|+1},
\end{align}
as loops are not in spanning trees. Hence if $e_j$ is a loop then $\int_0^{1}q_j=0$, and the sum of the roots of the second factor must also be zero. This is $(1-r)\sum_{e_j\textrm{is not a loop}}\int_0^{1}q_j$ multiplied by the number of spanning trees. Thus $\sum_{e_j\in G}\int_0^{1}q_j=0$ and we can proceed as in the proof of Theorem \ref{equallengthinverse}.
\qed

\end{section}

\bibliographystyle{amsplain}

\end{document}